\documentclass[smallextended,envcountsect]{my-svjour3} 

\smartqed 

\usepackage{amsmath,amssymb,amsfonts}
\usepackage{graphicx}
\usepackage{subfig}
\usepackage{url}
\usepackage[english]{babel}
\usepackage{calrsfs}
\usepackage{a4wide}
\usepackage{cite}
\usepackage{enumerate}
\usepackage[misc,geometry]{ifsym} 
\usepackage{hyperref}

\journalname{JOTA}

\newcommand{\doi}[1]{{\scriptsize 
\textsc{doi}: \href{http://dx.doi.org/#1}{\nolinkurl{#1}}}}

\newcommand*{\arXiv}[1]{\href{http://arxiv.org/abs/#1}{\texttt{arXiv:#1}}}

\def\R{\mathbb{R}}

\begin{document}

\title{Optimal Control of Aquatic Diseases: A Case Study of Yemen's Cholera Outbreak}

\author{Ana P. Lemos-Pai\~{a}o \and Cristiana J. Silva\\ 
Delfim F. M. Torres \and Ezio Venturino}

\institute{Ana P. Lemos-Pai\~{a}o \at
Center for Research and Development in Mathematics and Applications (CIDMA),\\
Department of Mathematics, University of Aveiro, 3810-193 Aveiro, Portugal\\
anapaiao@ua.pt
\and
Cristiana J. Silva \at
Center for Research and Development in Mathematics and Applications (CIDMA),\\
Department of Mathematics, University of Aveiro, 3810-193 Aveiro, Portugal\\
cjoaosilva@ua.pt
\and	
Delfim F. M. Torres \Letter \at
Center for Research and Development in Mathematics and Applications (CIDMA),\\
Department of Mathematics, University of Aveiro, 3810-193 Aveiro, Portugal\\
delfim@ua.pt
\and	
Ezio Venturino \at
Department of Mathematics ``Giuseppe Peano'',\\
Universit\`a di Torino, via Carlo Alberto 10, 10123, Torino, Italy\\
Member of the INdAM research group GNCS\\
ezio.venturino@unito.it}

% ------------------------------------------

\date{Received: April 6, 2019; Revised: October 12, 2019 and January 21, 2020; Accepted: April 15, 2020}

\dedication{Communicated by Alberto d'Onofrio}

\maketitle

% ------------------------------------------

\begin{abstract}
We propose a mathematical model for the transmission dynamics 
of some strains of the bacterium {\it{Vibrio cholerae}}, 
responsible for the cholera disease in humans. We prove that, 
when the basic reproduction number is equal to one, 
a transcritical bifurcation occurs for which the endemic equilibrium 
emanates from the disease-free point. A control function is introduced
into the model, representing the distribution 
of chlorine water tablets for water purification.   
An optimal control problem is then proposed and analyzed, where 
the goal is to determine the fraction of susceptible individuals 
who should have access to chlorine water tablets in order to minimize the 
total number of new infections plus the total cost associated 
with the distribution of chlorine water tablets, over the considered period of time.
Finally, we consider real data of the cholera outbreak 
in Yemen, from 27 April 2017 to 15 April 2018,
choosing the values of the parameters of the
uncontrolled model that fit the real data. Using our optimal
control results, we show, numerically, that the distribution of chlorine water tablets 
could have stopped, in a fast way, the worst cholera outbreak 
that ever occurred in human history. Due to the critical situation of Yemen, 
we also simulate the case where only a small percentage of susceptible individuals 
has access to chlorine water tablets and obtain an optimal control solution that 
decreases, substantially, the maximum number 
of infective individuals attained at the outbreak.
\end{abstract}



% ----------------------------------------------

\section{Introduction}

Cholera is an acute diarrhoeal illness caused by infection of
the intestine with some strains of the bacterium \textit{Vibrio cholerae}, 
which lives in an aquatic medium. Cholera remains a global threat to public 
health and an indicator of inequity and lack of social development 
\cite{WHO_cholera,LPST}. Cholera is a disease of poverty and
closely linked to poor sanitation and a lack of clean
drinking water \cite{WHO:cholera:vaccines}.
The ingestion of contaminated water can cause cholera outbreaks, 
as John Snow proved, in 1854 \cite{Shuai}. This is a way 
of transmission of the disease, but other ones exist. 
For example, susceptible individuals can become infected 
if they come in contact with infectious individuals. If individuals 
are at an increased risk of infection, then they can transmit 
the disease to other persons that live with them by sharing 
food preparation or water storage containers \cite{Shuai}. 
An individual can be infected with or without symptoms.
Some symptoms are watery diarrhoea, vomiting, and leg cramps. 
If an infective individual does not have treatment, 
then he becomes dehydrated, suffering of acidosis and circulatory
collapse. This situation can lead to death within 12 to 24 hours \cite{Mwasa,Shuai}.
Some studies and experiments suggest that a recovered individual can be immune 
to the disease during a period of 3 to 10 years. However, recent discoveries suggest 
that immunity can be lost after a period of weeks to months \cite{Neilan,Shuai}.
 
Since 1979, several mathematical models for the 
transmission of cholera have been proposed: see, e.g.,
\cite{Capasso:1979,Capone,Codeco:2001,Hartley:2006,Hove-Musekwa,%
Joh:2009,LPST,LPaiaoSilvaTorresYemen2018,Mukandavire:2008,Mwasa,Neilan,Pascual,Shuai} 
and references cited therein. In \cite{Neilan}, the authors propose
an SIR (Susceptible--Infectious--Recovered)
type model and consider two classes of bacterial concentrations 
(hyperinfectious and less-infectious) and two classes 
of infective individuals (asymptomatic and symptomatic). 
In \cite{Shuai}, another SIR-type model is proposed 
that incorporates hyperinfectivity 
(where infectivity varies with the time elapsed
since the pathogen was shed)
and temporary immunity, using distributed delays. 
The authors of \cite{Mwasa} incorporate in a SIR-type model 
public health educational campaigns, vaccination, quarantine and treatment, 
as control strategies in order to curtail the disease. 

The use of quarantine for controlling epidemic diseases 
has always been controversial, because such strategy raises political, 
ethical, and socioeconomic issues and requires a careful balance between 
public interest and individual rights \cite{Tognotti:quarantine}. 
Quarantine was adopted as a mean of separating persons, animals, 
and goods, which may have been exposed to a contagious disease. Since 
the fourteenth century, quarantine has been the cornerstone 
of a coordinated disease-control strategy, including isolation, 
sanitary cordons, bills of health issued to ships, fumigation, 
disinfection, and regulation of groups of persons who were believed 
to be responsible for spreading the infection \cite{Matovinovic,Tognotti:quarantine}.
The World Health Organization (WHO) does not recommend quarantine measures 
and embargoes on the movement of people and goods for cholera. However, 
cholera is still on the list of quarantinable diseases of the EUA National 
Archives and Records Administration \cite{CenterDiseaseControl}.

In this paper, we propose a SIQR (Susceptible--Infectious--Quarantined--Recovered)
type model based on \cite{LPST}, where it is assumed that infective individuals 
are subject to quarantine during the treatment period. 
We refine the model of \cite{LPST}, considering the important situation 
related to the fact that to become infected, a healthy individual 
must intake bacteria from the environment and, by doing it, 
these bacteria are removed from the aquatic environment. 
The model here proposed improves the one in \cite{LPST}, 
since the removal of the ingested bacteria from the environment
by susceptible individuals was not previously considered 
and it must be assumed. Our aim is to discover what happens 
when an efficient strategy through quarantine is implemented.

The consequences of a humanitarian crisis, such as disruption 
of water and sanitation systems or the displacement of populations 
to inadequate and overcrowded camps, can increase the risk 
of cholera transmission \cite{WHO_cholera}. The number of cholera cases 
reported by WHO has continued to be high over the last few years. 
During 2016, $132121$ cases were notified from 38 countries, 
including 2420 deaths \cite{WHO_cholera}. Recently, 
in Yemen the largest outbreak of cholera in the history of the world 
has occurred, \cite{TelegraphNews}. The epidemic began in October 2016 
and in February--March 2017 was in decline. However, on 27 April 2017, 
the epidemic broke out again. This happened ten days after Sanaa's 
sewer system had stopped working. Problems in infrastructures, health, 
water and sanitation systems in Yemen, allowed the fast spread of 
the disease \cite{wikipedia_cholera_yemen}. Between 27 April 2017 
and 15 April 2018, there were 1 090 280 suspected cases 
reported and 2 275 deaths due to cholera \cite{WHO_15abril2018}.
In \cite{Nishiura}, this outbreak is studied mathematically, 
forecasting the cholera epidemic in Yemen and explicitly 
addressing the reporting delay and ascertainment bias. 
On the other hand, in \cite{LPaiaoSilvaTorresYemen2018}  
a SIQRV (Susceptible--Infectious--Quarantined--Recovered--Vaccinated) type model 
is proposed, considering vaccination of susceptible individuals and describing well 
the cholera outbreak in Yemen, between 27 April 2017 and 15 April 2018.
In \cite{Sardar}, a compartmental model with periodic environment
is proposed, using a real-life data set of cholera epidemic for Zimbabwe, between 2008 and 2011.

Optimal control is a branch of mathematics developed
to find optimal ways to control a dynamical system
\cite{Cesari_1983,Fleming_Rishel_1975,Pontryagin_et_all_1962}.
There are few papers that apply optimal control to cholera models, see e.g. \cite{Neilan,LPST,Sardar}. 
Here we propose and analyze one such optimal control problem, 
where the control function represents the fraction of susceptible 
individuals $S$ who has access to chlorine water tablets (CWT) 
for water purification. Therefore, they are protected from infection. 
The objective is to find the optimal strategy through 
the use of CWT that minimizes the total number 
of new infections plus the total cost associated with 
CWT interventions. CWT are effervescent tablets that kill 
micro-organisms in water to prevent cholera, typhoid, dysentery, 
and other water borne diseases. There are different sizes of CWT 
and each tablet size is formulated to treat a specific volume 
of water, ranging from 1 liter to 2 500 liters. 

We prove that the extremal controls, 
derived from the Pontryagin Maximum Principle, satisfy the so-called 
strict bang-bang property \cite{Osmolovskii}.
Through numerical simulations, we show that the strategy 
given by the solution of the optimal control problem could have stopped, 
in a short time, the cholera outbreak on Yemen. 
Moreover, we simulate the cases of low, sufficient, and abundant resources, 
finding the interval of time needed to stop the outbreak in Yemen. 

The paper is organized as follows. In Section~\ref{Sec:model}, 
we propose a model for cholera transmission dynamics.
We analyze, in Section~\ref{Sec:mod:analysis}, the positivity 
and boundedness of the solutions, as well as the local stability 
and feasibility of the disease-free and endemic equilibria. 
In Section~\ref{sec:ocp}, we propose and analyze an optimal 
control problem for the minimization of the number of new infections 
through the distribution of CWT. 
Section~\ref{sec:num:simu} is devoted to numerical simulations 
and a case study in Yemen. The concluding Section~\ref{sec:conc}
discusses the optimal impact of CWT distribution on the control 
of the cholera outbreak in Yemen, pointing also
some directions for future work.

% ----------------------------------------------

\section{Model Setup}
\label{Sec:model}

The model contains the fundamental populations, identified as follows.
The humans are subdivided into the susceptible $S$, infective $I$, quarantined $Q$,
and recovered $R$. Then we also consider the free bacteria population living in the environment,
$B$. This is an important specification, as to become infected, a healthy individual
must intake bacteria from the environment and, in doing so, these bacteria are removed
from the aquatic medium. This feature, absent in \cite{LPST}, must be incorporated in the model,
to have a meaningful formulation. The model equations read as follows:
\begin{equation}
\label{ModeloColera}
\begin{cases}
\begin{split}
\dot{S}(t)=\ &\Lambda-\displaystyle\frac{\beta B(t)S(t)}{\kappa+B(t)}+\omega R(t)-\mu S(t)=: f_1,\\[0.3 cm]
\dot{I}(t)=\ &\displaystyle\frac{\beta B(t)S(t)}{\kappa+B(t)}-(\delta+\alpha_1+\mu)I(t)=: f_2,\\[0.3 cm]
\dot{Q}(t)=\ &\delta I(t)-(\varepsilon+\alpha_2+\mu)Q(t)=: f_3,\\[0.3 cm]
\dot{R}(t)=\ &\varepsilon Q(t)-(\omega+\mu)R(t)=: f_4,\\[0.3 cm]
\dot{B}(t)=\ &\eta I(t)-dB(t) - \displaystyle\frac{\rho B(t)S(t)}{\kappa+B(t)} =: f_5.
\end{split}
\end{cases}
\end{equation}
It should be noted that the new term 
for the removal of bacteria from the environment, 
through their uptake by susceptible individuals,	
here added with respect to \cite{LPST}, is very relevant.
Using the values for the parameters of Section~\ref{sec:num:simu},
Figure~\ref{fig1:rv2} quantifies the new term for the removal of bacteria 
from the environment along time. As it can be seen in this figure, 
the new term is quantitatively important.
% ------------------------------------------------------
\begin{figure}[!htb]
\centering
\includegraphics[scale=0.4]{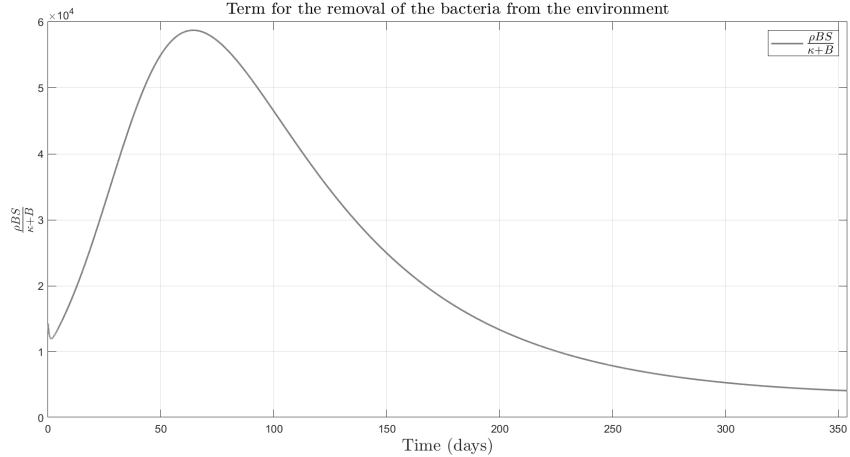}
\begin{minipage}{0.7\textwidth}
\caption{\small New term for the removal of bacteria from the environment 
through their uptake by susceptible individuals.}
\label{fig1:rv2}
\end{minipage}
\end{figure}
% ------------------------------------------------------
We could also consider that infected individuals also ingest contaminated
water. This would account for the modification of the last/bacteria equation,
by including additional terms $-aBI/(\kappa + B)$ and $-bBR/(\kappa + B)$.
However, we verified that these new terms are not quantitatively important.
	
The first equation describes the evolution of the healthy people. We assume to have an input
of new individuals at constant rate $\Lambda$, that can get infected by ingestion of bacteria
from the environment. This mechanism is expressed by the second term, with a saturating response function
in terms of the bacteria. It behaves linearly at first, as the more bacteria are acquired,
the higher the probability of getting the infection.
But this holds up to a certain point, because if the water already ingested is full of germs, to drink more of
it, will not substantially increase the threat to the health of the consumer, which is already quite endangered.
Therefore, the infection contagion level will saturate at rate $\beta$.
Note that more bacteria to be ingested to start the infection 
is already embedded in this parameter. The remaining terms in the equation represent the individuals that, 
after recovery, become susceptible again, at rate $\omega$, and the natural human mortality, at rate $\mu$.

The second equation for the infective contains the new recruitments, coming from the healthy
individuals as described above, and then the losses of this class, namely, the migration to the 
quarantined class, at rate $\delta$, the disease-related mortality, at rate $\alpha_1$, 
and the natural mortality, at rate $\mu$.
Although the assumption that all infected individuals enter the quarantine stage
may be a far-fetched assumption, if such ambitious studies are not carried out, 
almost no one will strive to improve recent and problematic endemic realities.
The uptake of viruses from bacteria in an aquatic environment is described in
\cite{MR1629476} (but see also \cite{MR2492346}
for a revisitation of the above model), a system that roughly corresponds to the first, 
second and last equation in the model at hand, where now bacteria take the place 
of viruses and people replace bacteria. Here, we further modify it considering that usually 
in ecological settings the quantity of food (or water)
consumption is usually assumed to be modeled via a Holling type II response function.
In turn, the mechanism describes the fact that the infection rate will saturate with the number of bacteria.

The quarantined class, third equation, incorporates the infective individuals who are so identified
at rate $\delta$. Then, it can lose people in three different ways: through migration to the recovered
class, at rate $\varepsilon$; still by disease-related mortality, but at a possibly different rate $\alpha_2$;
or, finally, by natural mortality, at rate $\mu$.

The recovered individuals come from the quarantined class at rate $\varepsilon$ and leave either by
returning susceptible at rate $\omega$ or by natural mortality. In this case, the people are not ill 
and, therefore, the disease-related mortality is not present.

The last equation accounts for the free bacteria in the environment, i.e., the water. 
In this medium they cannot survive, their mortality rate being $d$.
Nevertheless, they are continuously reversed into it by the infective individuals $I$, at rate $\eta$. 
Indeed, within the body of the infective, they do reproduce, 
this indeed being the cause of the illness, and are then released in the
open environment. A similar phenomenon could occur for the
quarantined individuals, who are still subject to the disease, but it is assumed that, as they are isolated
and treated in the hospitals, measures are taken so that they cannot propagate the infection. In particular,
they are prevented from fouling the water with new bacteria coming from the dejections of their bodies.
The third term in the equation describes, as mentioned before stating the model, the uptake of bacteria
from the water by healthy individuals during the infection process. 
We can easily assume that quarantined people are taken care of and therefore not exposed to and ingest
contaminated water. It is also reasonable to assume that, after recovery,
individuals would be very careful about the water they consume, in view of the fact
that during the treatment period they would be informed about the source of the disease.
We also assume that there is a prompt response of the authorities to infection through treatment (quarantine) 
so that the amount of bacteria ingested by infected individuals is residual and can be neglected.
Therefore, we can assume that bacteria are not taken up from the environment by these
classes of individuals. On the other hand, as already stated, they are essential in rendering
susceptible individuals diseased.

In a more abstract setting, we can rewrite \eqref{ModeloColera} 
by introducing the vector of the time-dependent variables,
\begin{equation}
\label{eq:def:X}
X = (x_1, x_2, x_3, x_4, x_5) = (S, I, Q, R, B),
\end{equation} 
for which \eqref{ModeloColera}, in compact form, becomes:
\begin{equation}
\label{ModeloColeraX}
\dot{x_i}(t)=f_i(x_1,\ldots,x_5), \quad i=1,\ldots,5.
\end{equation}

% ----------------------------------------------

\section{Analysis of the Model}
\label{Sec:mod:analysis}

Our analysis of \eqref{ModeloColera} follows the pattern used in \cite{LPST}.

% --------------------

\subsection{Preliminary Results}
\label{subsec:preliminary_results}

Assuming the ecologically meaningful nonnegative initial conditions for the populations,
the solutions of the dynamical system remain nonnegative for all time. This result is
contained in Lemma~1 of \cite{LPST} and translates in this situation without any change.
The solutions not only remain in the positive cone, but are also bounded and the positively
invariant set $\Omega$ is the same already found in \cite{LPST}, namely:
\begin{equation}
\label{eq:Omega}
\begin{gathered}
\Omega = \Omega_H \times \Omega_B, \qquad 
\Omega_B = \left\{ B \in \R_0^+ :  B(t)  
\leq \frac{\Lambda\eta}{\mu d} \right\}, \\ 
\Omega_H = \left\{ (S, I, Q, R) \in \left(\R_0^+\right)^4 : 
S(t) + I(t) + Q(t) + R(t) \leq \frac{\Lambda}{\mu} \right\}.
\end{gathered}
\end{equation}
The proof is essentially the same, with only a change in the derivation
of the upper bound for the bacteria population, which is obtained
by dropping the last term of the last equation in \eqref{ModeloColera}
to obtain the same bound:
\begin{equation*}
\dot{B}(t) \leq \eta I(t)- d B(t) \leq \eta \frac{\Lambda}{\mu} - d B(t) \, . 
\end{equation*}

% --------------------

\subsection{System's Equilibria}
\label{sec:3.2}

The only possible equilibria of model \eqref{ModeloColera} 
are the disease-free point (DFE) and coexistence, 
or the endemic equilibrium (EE), as in \cite{LPST}.
However, we will repeat here the analysis in some detail 
as it entails some relevant differences.

As for the DFE, we find again
\begin{equation}
\label{eq:DFE}
E^0=(S^0,I^0,Q^0,R^0,B^0)=\left(\frac{\Lambda}{\mu},0,0,0,0\right).
\end{equation}
The basic reproduction number $R_0$ can then be evaluated, 
following \cite{Mwasa,Driessche}.

\begin{proposition}
\label{prop:R0}
The basic reproduction number of model \eqref{ModeloColera} is
\begin{equation}
\label{R0}
R_0
= \dfrac {\beta\Lambda\eta}{(\delta+\alpha_1+\mu)\, 
\left( \mu\kappa d+\rho\,\Lambda \right) }.
\end{equation}
\end{proposition}

\begin{proof}
Follows easily using the methods described in \cite{Diekmann}.
\hfill $\Box$
\end{proof}

The basic reproduction number proves to be instrumental in the local
stability issue of the DFE $E^0$, as it is shown in the next result.

\begin{theorem}
The disease-free equilibrium $E^0$ of model \eqref{ModeloColera} is
locally asymptotic stable if 
\begin{equation}
\label{DFE_stab}
R_0<1.
\end{equation}
\end{theorem}

\begin{proof}
Let us write the right-hand side of system 
\eqref{ModeloColera} as $\mathcal{F} - \mathcal{V}$ 
according to the approach of \cite{Driessche}.
The characteristic equation of \eqref{ModeloColera} evaluated at the DFE,
$p(\chi)=\det(F_0-V_0-\chi I_5)$ with $F_0$ and $V_0$ 
the Jacobian matrices of $\mathcal{F}$ and $\mathcal{V}$ computed 
at the disease free equilibrium $E^0$, respectively, 
factorizes to produce three explicit 
eigenvalues, $-\mu<0$, $-a_2<0$, and $-a_3<0$,
and a quadratic equation in $\chi$:
$$
(a_1+\chi )\left( d+ \dfrac{\rho \Lambda}{\mu\kappa} +\chi \right) 
- \displaystyle\frac{\beta \Lambda \eta}{\mu\kappa} =0,
$$
for which the Routh--Hurwitz conditions are easily seen to hold if
$$
a_1 \left( d+ \dfrac{\rho \Lambda}{\mu\kappa} \right) 
- \displaystyle\frac{\beta \Lambda \eta}{\mu\kappa} > 0,
$$
which amounts to the condition \eqref{DFE_stab}.
\hfill $\Box$
\end{proof}

A converse result holds in case of the opposite condition, 
as illustrated in the following result.

\begin{proposition}
\label{prop:EE}
Let $a_1=\delta+\alpha_1+\mu$, 
$a_2=\varepsilon+\alpha_2+\mu$, 
and $a_3=\omega+\mu$.
Assume that $\lambda^*,\ \delta,\ \varepsilon,\ \omega>0$. 
If $R_0>1$, then model \eqref{ModeloColera} has the endemic equilibrium
\begin{equation}
\label{EndemicEquilibrium}
E^*=(S^*,I^*,Q^*,R^*,B^*)
=\left(\frac{\Lambda a_1a_2a_3}{D},\frac{\Lambda a_2a_3\lambda^*}{D},
\frac{\Lambda\delta a_3\lambda^*}{D},
\frac{\Lambda \delta\varepsilon\lambda^*}{D}, (\beta \eta - \rho a_1)
\frac{\Lambda a_2a_3\lambda^*}{\beta Dd}\right),
\end{equation}
where 
\begin{equation}
\label{D}
D=a_1a_2a_3(\lambda^*+\mu)-\delta\varepsilon\omega\lambda^*, 
\quad \lambda^*= \beta B^* (\kappa+B^*)^{-1},
\end{equation}
which is feasible if
\begin{equation}
\label{EE_feas}
\beta \eta > \rho a_1.
\end{equation}
\end{proposition}

\begin{proof}
For this equilibrium to be feasible, the transmission rate must be strictly positive:
$$
\beta B^*(t)\left(\kappa+B^*(t)\right)^{-1}>0. 
$$
Solving in turn the second, third, and fourth equilibrium equation of \eqref{ModeloColera},
we find
$$
S^*=\dfrac {a_1}{\lambda^*} I^*, \quad I^*=\dfrac {a_2}{\delta} Q^*, 
\quad Q^*=\dfrac {a_3}{\varepsilon} R^*.
$$
Then, we obtain $$S^*= \dfrac {a_1a_2a_3}{\lambda^*\delta\varepsilon} R^*.$$
Substituting the last evaluated value of $S^*$ into the first equilibrium equation, 
we then obtain $\delta \varepsilon \lambda^* \Lambda - DR^*=0$, 
which gives the fourth component of \eqref{EndemicEquilibrium},
and by back substitution also the first three. Finally, the fifth 
equation provides the value of $B^*$, which must be nonnegative 
to be feasible, giving thus \eqref{EE_feas}. Now, from 
$\lambda^*= \beta B^* (\kappa+B^*)^{-1}$, 
substituting the value of $B^*$ and rearranging, we obtain
$$
\{[\Lambda (\beta \eta - \rho a_1)+\kappa \beta d a_1] a_2a_3 
- \kappa \beta d\delta\varepsilon\omega \}\lambda^*
=[\beta\Lambda\eta -a_1 (\rho \Lambda + \mu \kappa d )] \beta a_2 a_3.
$$
It follows that
$$
\lambda^*
=\dfrac { \beta ( R_0 - 1 ) a_1 a_2a_3 (\rho\Lambda + \mu \kappa d )} 
{\Lambda (\beta\eta -\rho a_1) a_2a_3+\kappa \beta d [a_1 a_2a_3 - \delta\varepsilon\omega ]}.
$$
In view of \eqref{EE_feas}, and the fact that
\begin{equation}
\label{a1a2a3}
a_1a_2a_3-\delta\varepsilon\omega=(\delta+\alpha_1+\mu)(\varepsilon
+\alpha_2+\mu)(\omega+\mu)-\delta\varepsilon\omega>0,
\end{equation}
because $\alpha_1, \alpha_2 \geq 0$ and all the other coefficients are positive,
the above value of $\lambda^*$ is positive if and only if $R_0>1$. In such case 
the model \eqref{ModeloColera} has the endemic equilibrium \eqref{EndemicEquilibrium}.
This concludes the proof.
\hfill $\Box$
\end{proof}

\begin{remark}
Comparing the local stability condition for the DFE and the feasibility
condition for the EE, it is easily seen that for $R_0=1$ a transcritical bifurcation
occurs for which the endemic equilibrium $E^*$ emanates from the disease-free
point $E^0$.
\end{remark}

\begin{remark}
Note that the feasibility result for $E^*$
differs from the corresponding one in \cite{LPST}
because for the epidemics to subsist, it is necessary that 
inequality \eqref{EE_feas} holds.
\end{remark}

We next consider the local stability issue of the EE. Although the final result coincides
with the one obtained in \cite{LPST}, there are some details that change. Thus, we present
also its proof, following the same steps and employing the abstract formulation
given by \eqref{eq:def:X} and \eqref{ModeloColeraX}. 

\begin{theorem}
The equilibrium points $E^0$ and $E^*$ of \eqref{ModeloColera} 
are, respectively, unstable and locally asymptotic stable for $R_0 > 1$.
\end{theorem}

\begin{proof}
We apply the method of \cite[Theorem~4.1]{Castillo}, 
choosing $\beta^*$ as bifurcation parameter. Finding its value from $R_0=1$, we have
\begin{align*}
\beta^* = 
\dfrac{a_1 (\rho \Lambda + \mu \kappa d)}{\Lambda\eta},
\end{align*}
which is positive in view of \eqref{EE_feas}. 
At $\beta^*$, the Jacobian of \eqref{ModeloColeraX} 
evaluated at $E^0$ becomes
\begin{align*}
J_0^*=\left[
\begin{matrix}
-\mu & 0 & 0 & \omega & -\displaystyle\frac{a_1(\rho\Lambda+\mu\kappa d)}{\eta\mu\kappa} \\
0 & -a_1 & 0 & 0 & \displaystyle\frac{a_1(\rho\Lambda+\mu\kappa d)}{\eta\mu\kappa} \\
0 & \delta & -a_2 & 0 & 0 \\
0 & 0 & \varepsilon & -a_3 & 0 \\
0 & \eta & 0 & 0 & -\displaystyle\frac{\rho \Lambda + \mu\kappa d }{\mu\kappa}
\end{matrix}
\right].
\end{align*}
Its eigenvalues are
$-a_1-\displaystyle\frac{\rho\Lambda+\mu\kappa d}{\mu\kappa}$,
$-a_2$, $-a_3$, $-\mu$, and $0$. 
Thus, zero is a simple eigenvalue of $J_0^*$ and,
recalling \eqref{EE_feas},
all the other eigenvalues have negative real parts. The center manifold
theory \cite{Carr} can thus be employed to assess the behavior of \eqref{ModeloColeraX} 
near $\beta=\beta^*$. The tool for studying the local asymptotic stability property
of the EE for $\beta$ near $\beta^*$ is provided by Theorem~4.1 of \cite{Castillo}.
The right and left eigenvectors associated with the
zero eigenvalue of $J_0^*$ are, respectively, 
$w=\left[
\begin{matrix}
w_1 & w_2 & w_3 & w_4 & w_5
\end{matrix}
\right]^T$ 
and $v=\left[
\begin{matrix}
v_1 & v_2 & v_3 & v_4 & v_5
\end{matrix}
\right]$, i.e., explicitly,
\begin{align*}
w&=\left[\begin{matrix}
\left(\displaystyle\frac{\delta\varepsilon\omega}{a_2a_3}
-a_1\right)\displaystyle\frac{1}{\mu} \ & \ 1 \ 
& \ \displaystyle\frac{\delta}{a_2} 
\ & \ \displaystyle\frac{\delta\varepsilon}{a_2a_3} 
\ & \ 
\displaystyle\frac{\mu\kappa\eta}{\rho\Lambda + \mu\kappa d}
\end{matrix}
\right]^Tw_2,
\quad
v=\left[\begin{matrix}
0 \ & \ 1 \ & \ 0 \ & \ 0 \ & 
\ \displaystyle\frac{a_1}{\eta}
\end{matrix}\right] v_2,
\end{align*}
with $w_2$ and $v_2$ arbitrary constants. 
So, we can choose $w_2 = v_2 = 1$. 
In view of the fact that $v_1=v_3=v_4=0$,
the only nonvanishing derivatives, in the above expressions, are
\begin{equation*}
\begin{split}
& \left[\frac{\partial^2f_2}{\partial x_1\partial x_5}\left(E^0\right)\right]_{\beta=\beta^*}
=\left[\frac{\partial^2f_2}{\partial x_5\partial x_1}\left(E^0\right)\right]_{\beta=\beta^*}
=\frac{\beta^*}{\kappa},\\
& \left[\frac{\partial^2f_2}{\partial x_5^2}\left(E^0\right)\right]_{\beta=\beta^*}
=-\frac{2\beta^*\Lambda}{\mu\kappa^2},\\
& \left[\frac{\partial^2f_5}{\partial x_1\partial x_5}\left(E^0\right)\right]_{\beta=\beta^*}
=\left[\frac{\partial^2f_5}{\partial x_5\partial x_1}\left(E^0\right)\right]_{\beta=\beta^*}
=-\frac{\rho}{\kappa},\\
& \left[\frac{\partial^2f_5}{\partial x_5^2}\left(E^0\right)\right]_{\beta=\beta^*}
=\frac{2\rho\Lambda}{\mu\kappa^2}.
\end{split}
\end{equation*}
Let us assume that $\varphi=\beta-\beta^*$.
Therefore, recalling \eqref{EE_feas} and \eqref{a1a2a3}, for constants $a$ and $b$, we find
\begin{equation*}
a=  \frac{2 \mu (\beta^* \eta - \rho a_1)}{\rho
\Lambda + \mu\kappa d}\left(\frac{\delta \varepsilon\omega 
- a_1 a_2 a_3}{a_2 a_3 \mu}-\frac{\Lambda \eta}{\rho\Lambda + \mu\kappa d}\right)
<0
\end{equation*}
and
\begin{equation*}
\begin{split}
b= &\ \sum_{i=1}^{5}\left(
v_2w_i\left[\frac{\partial^2f_2}{\partial x_i\partial\varphi}\left(E^0\right)\right]_{\beta=\beta^*}
\right)\\
= &\ v_2w_5\left[\frac{\partial}{\partial x_5}
\left(\frac{x_1x_5}{\kappa+x_5}\right)\left(E^0\right)\right]_{\beta=\beta^*} \\
=&\ \frac{\Lambda\eta}{\rho\Lambda + \mu\kappa d} >0,
\end{split}
\end{equation*}
respectively. Thus, since $\beta^*\eta > \rho a_1$, as 
\begin{equation*}
\begin{cases}
a<0\\
b>0\\
\varphi=\beta-\beta^*=
\dfrac{a_1(\rho\Lambda+\mu\kappa d)}{\Lambda\eta}(R_0-1)>0
\end{cases}
\Leftrightarrow
\begin{cases}
a<0\\
b>0\\
R_0>1,
\end{cases}
\end{equation*}
we conclude from Theorem~4.1 of \cite{Castillo}
that the equilibrium points $E^0$ and $E^*$ of \eqref{ModeloColera} 
are, respectively, unstable and locally asymptotic stable
for a value of the basic reproduction number such that $R_0>1$. 
This concludes the proof.
\hfill $\Box$
\end{proof}

%---------------------------------------------------------------------------

\section{The Optimal Epidemic Control}
\label{sec:ocp}

In this section, we define an optimal control problem with the purpose 
to curtail the spread of the epidemic. Furthermore, we write the respective 
necessary optimality conditions, following the Pontryagin approach
\cite{Pontryagin_et_all_1962}. We keep on using the notation
$X=(x_1,x_2,x_3,x_4,x_5)=(S,I,Q,R,B)$.

% --------------------

\subsection{The Optimal Control Problem}
	
Cholera transmission is linked to inadequate access to clean water 
and sanitation facilities. The distribution of CWT 
for water purification is one of the possible strategies to improve 
the quality of the water and control cholera outbreaks.  
In this section, we introduce into model \eqref{ModeloColera} 
a control function $u(\cdot)$ that represents the fraction of susceptible 
individuals who has access to CWT for water purification 
(see \cite{IntMedCorps_tablets_cholera}). This control measure is such that 
$u(t)\in [0, u_{\max}]$ for all $t\in[0,T]$, where $T>0$ is the final time 
and $0 \leq u_{\max} \leq 1$. If $u = 0$, then nobody receives those chlorine 
water tablets, that is, there is no control measure. 
To assume that only a fraction of susceptible 
receive the water tablets (e.g., those living in areas that 
can be reached by health workers)  
while the rest of the population does not receive sanitary aid, 
we fix a value of $0<u_{\max} < 1$. If $u = 1$, then there is no movement
of individuals from class $S$ to class $I$, i.e., there is no new infections.
The model with control is then given by the following system 
of non-linear ordinary differential equations:
\begin{equation}
\label{ModeloColeraControlo}
\begin{cases}
\begin{split}
\dot{S}(t)=\ &\Lambda-\displaystyle\frac{\beta B(t)S(t)}{\kappa+B(t)}\Big(1-u(t)\Big )
+\omega R(t)-\mu S(t)=\tilde{f_1}\big (X(t),u(t)\big ),\\[0.3 cm]
\dot{I}(t)=\ &\displaystyle\frac{\beta B(t)S(t)}{\kappa+B(t)}\Big(1-u(t)\Big)
-(\delta +\alpha_1+\mu)I(t)=\tilde{f_2}\big (X(t),u(t)\big ),\\[0.3 cm]
\dot{Q}(t)=\ &\delta I(t)-(\varepsilon+\alpha_2+\mu)Q(t)
=\tilde{f_3}\big (X(t),u(t)\big ),\\[0.3 cm]
\dot{R}(t)=\ &\varepsilon Q(t)-(\omega+\mu)R(t)=\tilde{f_4}\big (X(t),u(t)\big ),\\[0.3 cm]
\dot{B}(t)=\ &\eta I(t)-dB(t) - \displaystyle\frac{\rho B(t)S(t)}{\kappa+B(t)}
=\tilde{f_5}\big (X(t),u(t)\big ),
\end{split}
\end{cases}
\end{equation}
together with the initial conditions given by
\begin{equation}
\label{eq:init:cond_delay_OC}
S(0) = S_0 \geq 0,\ I(0) = I_0 \geq 0,\ Q(0)= Q_0 \geq 0,\ 
R(0) = R_0 \geq 0\ \textrm{and}\ B(0) = B_0 \geq 0. 
\end{equation}
The set $\mathcal{X}$ of admissible trajectories 
and the admissible control set $\mathcal{U}$ are, respectively, given by
\begin{eqnarray*}
\mathcal{X} = \left\{X(\cdot) \in W^{1,1}\left([0,T];\R^5\right) :   
\eqref{ModeloColeraControlo} \textrm{ and } \eqref{eq:init:cond_delay_OC} 
\textrm{ are satisfied}\right\},\\
\mathcal{U} = \big \{ u(\cdot) \in L^1\left([0, T]; \mathbb{R}\big ) 
: 0 \leq u (t) \leq u_{\max},  \, \forall \, t \in [0, T] \, \right\},
\end{eqnarray*}
where $0 \leq u_{\max} \leq 1$. 
The functional to be minimized is represented by
\begin{equation}
\label{costfunction}
J\big (X(\cdot),u(\cdot)\big ) = 
\int_{0}^{T}\frac{\beta B(t) S(t)}{\kappa + B(t)}(1-u(t))dt+c\int_{0}^{T}u(t)dt,
\end{equation}
i.e., the total number of new infections over the period, plus total cost of interventions,
where $c$ is a weight coefficient.
Clearly, one would like to eradicate
the epidemic at the least possible cost. Thus, the optimal control 
problem consists of determining the vector function
$$
X^\diamond(\cdot) 
= 
\big(x_1^\diamond(\cdot), x_2^\diamond(\cdot), x_3^\diamond(\cdot), x_4^\diamond(\cdot), x_5^\diamond(\cdot) \big)
=
\big(S^\diamond(\cdot), I^\diamond(\cdot), Q^\diamond(\cdot), R^\diamond(\cdot), B^\diamond(\cdot) \big)
\in \mathcal{X}
$$ 
associated with an admissible control $u^\diamond(\cdot) \in \mathcal{U}$ 
on the time interval $[0, T]$, that provides the minimal value 
to the cost functional \eqref{costfunction}, \textrm{i.e.},
\begin{equation}
\label{mincostfunct}
J\big(X^\diamond(\cdot),u^\diamond(\cdot)\big) 
= \min_{(X(\cdot),u(\cdot))\in\mathcal{X}\times\mathcal{U}} 
J\left(X(\cdot),u(\cdot)\right).
\end{equation}

% --------------------

\subsection{Necessary Optimality Conditions}

The following theorem provides necessary optimality conditions for
the optimal control problem \eqref{ModeloColeraControlo}--\eqref{mincostfunct}, 
assuming existence of solution.

\begin{theorem}
\label{Cap7:theo:nc}	
Assume that 
$
X^\diamond(\cdot) = 
\big(x_1^\diamond(\cdot), x_2^\diamond(\cdot), x_3^\diamond(\cdot), 
x_4^\diamond(\cdot), x_5^\diamond(\cdot) \big)\in\mathcal{X}
$
is an optimal state associated with the optimal control 
$u^\diamond(\cdot)\in\mathcal{U}$
of problem \eqref{ModeloColeraControlo}--\eqref{mincostfunct} 
with fixed final time $T\in\mathbb{R}_+$.
Then, there is a multiplier function 
$\lambda^\diamond = 
\left(\lambda_1^\diamond, \lambda_2^\diamond, \lambda_3^\diamond, 
\lambda_4^\diamond, \lambda_5^\diamond\right) 
: [0, T] \to \mathbb{R}^5$  
that satisfies the adjoint system
\begin{equation}
\label{adjoint_function:Ezio}
\begin{cases}
\begin{split}
\dot{\lambda}_1^\diamond(t)
=& \frac{x_5^\diamond(t)}{\kappa+x_5^\diamond(t)}
\Big(\beta \big[\lambda_1^\diamond(t)-\lambda_2^\diamond(t)-1\big]\big[1-u(t)\big] 
+ \rho \lambda_5^\diamond(t)\Big )+\mu\lambda_1^\diamond(t),\\[0.2 cm]
%--------------------
\dot{\lambda}_2^\diamond(t)
=& \Big (\delta + \alpha_1 + \mu \Big )\lambda_2^\diamond(t) 
- \delta\lambda_3^\diamond(t) 
- \eta\lambda_5^\diamond(t),\\[0.2 cm]
%--------------------
\dot{\lambda}_3^\diamond(t)
=& \Big (\varepsilon+\alpha_2+\mu\Big )\lambda_3^\diamond(t) 
-\varepsilon\lambda_4^\diamond(t),\\[0.2 cm]
%--------------------
\dot{\lambda}_4^\diamond(t)
=& -\omega\lambda_1^\diamond(t) 
+ \Big(\omega+\mu\Big)\lambda_4^\diamond(t),\\[0.2 cm]
%--------------------
\dot{\lambda}_5^\diamond(t)
=& \frac{\kappa x_1^\diamond(t)}{(\kappa+x_5^\diamond(t))^2}\Big(
\beta \big [\lambda_1^\diamond(t)-\lambda_2^\diamond(t)-1\big ]\big [1-u(t)\big] 
+ \rho \lambda_5^\diamond(t)\Big) + d\lambda_5^\diamond(t),
\end{split}
\end{cases}
\end{equation}
with transversality conditions 
\begin{equation}
\label{eq:TC:Ezio}
\lambda_i^\diamond(T)=0, \quad i=1,\ldots,5,
\end{equation}
for almost all $t\in[0,T]$.
Moreover, the control law is characterized by
\begin{equation}
\label{control_law:Ezio}
u^\diamond(t)=
\begin{cases}
\begin{split}
& u_{\max}, 	 \ && \text{if} \quad \phi(t)<0,\\
& 0, 			 \ && \text{if} \quad \phi(t)>0,\\
&\text{singular},\ &&\text{if} \quad \phi(t)=0\ \text{on}\ I_s\subset[0,T],
\end{split}
\end{cases}
\end{equation}
where $\phi$ is the switching function defined by 
\begin{equation}
\label{switching-function:Ezio}
\phi(t)=
c+\frac{\beta x_1^\diamond(t)x_5^\diamond(t)}{\kappa
+ x_5^\diamond(t)}\Big(\lambda_1^\diamond(t)-\lambda_2^\diamond(t)-1\Big)
\end{equation}
for almost all $t\in[0,T]$.
\end{theorem}

\begin{proof}
Necessary optimality conditions for 
\eqref{ModeloColeraControlo}--\eqref{mincostfunct} are given by 
Pontryagin's Minimum Principle of optimal control 
(see \cite{Pontryagin_et_all_1962}).
The Hamiltonian function is defined by
\begin{equation}
\label{Hamiltonian:Ezio}
H(X, u, \lambda) = 
\frac{\beta x_1 x_5}{\kappa + x_5}(1-u)+c u
+ \sum_{i=1}^{5} \lambda_i \tilde{f}_i(X,u).
\end{equation}
Let us suppose that 
$\big (X^\diamond(\cdot),u^\diamond(\cdot)\big )\in\mathcal{X}\times\mathcal{U}$ 
is an optimal solution of \eqref{ModeloColeraControlo}--\eqref{mincostfunct} 
with fixed final time $T\in\mathbb{R}_+$.
Then, 
there is an adjoint function 
$\lambda^\diamond 
= \big (\lambda_1^\diamond, \lambda_2^\diamond, \lambda_3^\diamond, 
\lambda_4^\diamond, \lambda_5^\diamond\big ) 
: [0, T] \to \mathbb{R}^5$, 
$\lambda^\diamond(\cdot) \in W^{1,1}\big ([0,T];\R^5\big )$, 
that satisfies, for almost all $t \in [0,T]$, the
\begin{enumerate}[1)]
\item transversality conditions:
\begin{equation}
\label{eq:TC:PMP:Ezio}
\lambda_i^\diamond(T) = 0, \quad i =1,\ldots, 5,
\end{equation}
in view of the free terminal state $X(T)$;
		
\item adjoint system:
\begin{equation}
\label{adjsystemPMP:Ezio}
\dot{\lambda}_i^\diamond(t)=
-\frac{\partial H}{\partial x_i}\big (X^\diamond(t),u^\diamond(t),\lambda^\diamond(t)\big ), 
\quad i =1,\ldots, 5;
\end{equation}
		
\item minimality condition:
\begin{equation}
\label{maxcondPMP:Ezio}
\min_{0 \leq u \leq u_{\max}}
H\big (X^\diamond(t), u, \lambda^\diamond(t)\big )
=
H\big (X^\diamond(t), u^\diamond(t), \lambda^\diamond(t)\big ),
\end{equation}
where $0 \leq u_{\max} \leq 1$
\end{enumerate}
So, conditions \eqref{eq:TC:Ezio} are derived from transversality conditions 
\eqref{eq:TC:PMP:Ezio}.
Moreover, system~\eqref{adjoint_function:Ezio} is obtained from 
adjoint system~\eqref{adjsystemPMP:Ezio}.
Let us evaluate the minimality condition \eqref{maxcondPMP:Ezio}. 
The Hamiltonian \eqref{Hamiltonian:Ezio} is linear in the control variable. 
Hence, the minimizer control is determined by the sign of the switching function
\begin{equation*}
\phi(t)= \frac{\partial H}{\partial u}\big (X^\diamond(t), u(t), \lambda^\diamond(t)\big )
= c+\frac{\beta x_1^\diamond(t)x_5^\diamond(t)}{\kappa
+ x_5^\diamond(t)}\Big(\lambda_1^\diamond(t)-\lambda_2^\diamond(t)-1\Big)
\end{equation*}
(see \eqref{switching-function:Ezio}) as follows:
\begin{equation*}
u^\diamond(t)=
\begin{cases}
\begin{split}
& u_{\max},        \ &&\text{if} \quad \phi(t)<0,\\
& 0,              \ &&\text{if} \quad \phi(t)>0,\\
& \text{singular},\ &&\text{if} \quad \phi(t)=0\ \text{on}\ I_s\subset[0,T].
\end{split}
\end{cases}
\end{equation*}
This concludes the proof.
\hfill $\Box$
\end{proof}

\begin{remark}
If the switching function has only finitely many isolated zeros 
in an interval $I_b \subset [0,T]$, then the
control $u^\diamond $ is called \textit{bang-bang} on $I_b$. 
The case of a \textit{singular control},
where $\phi(t) = 0$ on $I_s \subset [0,T]$, was not
discussed in the proof of Theorem~\ref{Cap7:theo:nc}, 
since in our computations we never encountered singular controls. 
\end{remark}

% ------------------------------------------------------

\section{Numerical Simulations and Discussion}
\label{sec:num:simu}

In this section, we show that the control measure defined 
in Section~\ref{sec:ocp} could have stopped more quickly 
the worst cholera outbreak that ever occurred in human history, 
which began in Yemen on April 27th, 2017. We consider 
the real data of the number of infective individuals in Yemen, 
from April 27th, 2017 to April 15th, 2018 \cite{WHO}, represented 
in Figure~\ref{fig:Yemen}. In this period, the 
maximum number of infective individuals was $51\ 000$.

\begin{remark}
The WHO tables report the number of cases per week, and these are not the same quantities
as our curve $I(t)$. There is, however, a (partial) justification for the comparison
carried out in Figure~\ref{fig:Yemen}, since $\delta$ (the quarantine rate) 
is larger than most other rates in the model (see Table~\ref{Tab_Parameter}). Indeed,
one may make a quasi-equilibrium assumption about $\dot{I}$, obtaining
$$
I(t) \approx \frac{C(t)}{\delta + \alpha_1+ \mu} \approx \frac{C(t)}{1.15},
$$
where $C(t)$ is the rate of new infections. Thus the number of cases per 
day should be close (though not equal) to $I(t)$.
\end{remark}

In order to better simulate real life situations, where there 
is a lack of resources needed to distribute CWT for water 
purification, we consider three situations: 
\begin{itemize}
\item low resources ($u_{\max}=0.20$); 
\item enough resources (two cases considered: 
$u_{\max}=0.55$ and $u_{\max}=0.90$); 
\item abundance of resources ($u_{\max}=0.95$). 
\end{itemize}

In the current section, we also observe numerically the local asymptotic stability of 
the endemic equilibrium, when we consider all values of Table~\ref{Tab_Parameter} and 
$T=5\times10^5$ days.

In Subsection~\ref{subsec:OC_low}, we consider low resources for CWT 
distribution ($u_{\max}=0.20$). It means that only a small percentage ($20\%$) 
of susceptible individuals has access to the CWT. In Subsection~\ref{subsec:OC_enough}, 
we consider sufficient enough resources
to decrease the outbreak (cases $u_{\max}=0.55$ and $u_{\max}=0.90$). 
Finally, in Subsection~\ref{subsec:OC_abundance}, we present numerical simulations 
with abundance of resources, that is, $u_{\max}=0.95$, for which almost all susceptible population 
has access to pure water. In all simulations, the parameters of model \eqref{ModeloColeraControlo} 
and initial conditions \eqref{eq:init:cond_delay_OC} are fixed at the values of
Table~\ref{Tab_Parameter}. Note that the initial values of the population 
and of the bacterial concentration satisfy \eqref{eq:Omega}, that is,
$$
N_0=S_0+I_0+Q_0+R_0\in\left[0,\frac{\Lambda}{\mu}\right] 
\quad \textrm{and} \quad B_0\in\left[0,\frac{\Lambda\eta}{\mu d}\right].
$$
Therefore, $(S_0,I_0,Q_0,R_0,B_0)\in\Omega=\Omega_H\times\Omega_B$. 
This implies that all the following numerical solutions $(S,I,Q,R,B)$ 
belong to the positively invariant set $\Omega=\Omega_H\times\Omega_B$ 
(see Subsection~\ref{subsec:preliminary_results} 
and Lemmas~1 and 2 of \cite{LPST}). In all the simulations,
we have assigned the value one to the cost weight associated with CWT distribution
($c = 1$). Note that we keep the same values for the parameters as the ones found 
in \cite{LPaiaoSilvaTorresYemen2018}. The only parameter that makes sense to change 
is the ingestion rate $\beta$, which is related with the novelty in the new model
we propose here. This value of $\beta$ was chosen in order to minimize the distance 
between real data and the curve of infected individuals 
predicted by model \eqref{ModeloColera}.
% ------------------------------------------------------
\begin{figure}[!htb]
\centering
\includegraphics[scale=0.47]{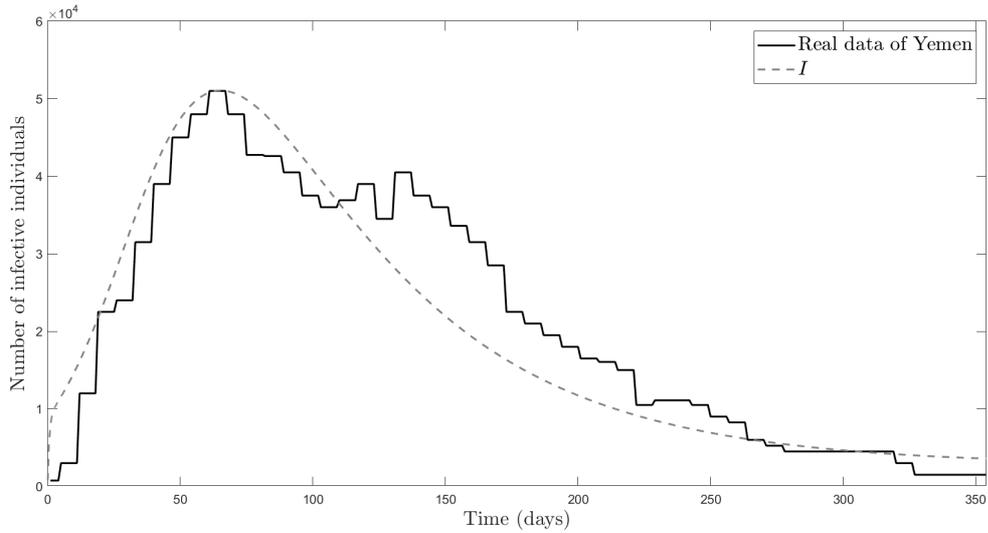}
\begin{minipage}{0.90\textwidth}
\caption{\small Number of infective individuals with cholera, per week, 
in Yemen, from 27 April 2017 to 15 April 2018 (see \cite{WHO_15abril2018}) 
versus state trajectory $I(t)$ for all $t \in [0, 354]$, predicted by
model \eqref{ModeloColera}, assuming all values of Table~\ref{Tab_Parameter}.}
\label{fig:Yemen}
\end{minipage}
\end{figure}
% ------------------------------------------------------
\begin{table}[ht]
\begin{minipage}{0.88\textwidth}
\caption{\small Parameter values and initial conditions 
for the optimal control problem \eqref{ModeloColeraControlo}--\eqref{mincostfunct}.}
\label{Tab_Parameter}
\end{minipage}
\begin{center}
\begin{tabular}[center]{|l|l|l|l|} \hline
\textbf{Parameter} & \textbf{Description} & \textbf{Value} & \textbf{Reference}\\ \hline \hline
\small $\Lambda$ & \small Recruitment rate & \small 28.4$N(0)$/365000 (person day$^{-1}$) & \small \cite{BirthRate}\\
\small $\mu$ & \small Natural death rate & \small 1.6$\times10^{-5}$ (day$^{-1}$) & \small \cite{DeathRate}\\
\small $\beta$ & \small Ingestion rate & \small 0.01891 (day$^{-1}$)  & \small Hypothetical\\
\small $\kappa$ & \small Half saturation constant & \small$10^7$ (cell/ml) & \small Hypothetical\\
\small $\omega$ & \small Immunity waning rate & \small 0.4/365 (day$^{-1}$) & \small \cite{Neilan}\\
\small $\delta$ & \small Quarantine rate & \small 1.15 (day$^{-1}$) & \small Hypothetical\\
\small $\varepsilon$ & \small Recovery rate & \small 0.2 (day$^{-1}$) & \small \cite{Mwasa}\\
\small $\alpha_1$ & \small Death rate (infected) & \small 6$\times10^{-6}$ (day$^{-1}$) & \small \cite{DeathRate,WHO}\\
\small $\alpha_2$ & \small Death rate (quarantined)& \small 3$\times10^{-6}$ (day$^{-1}$) & \small Hypothetical\\
\small $\eta$ & \small Shedding rate (infected) & \small 10 (cell/ml day$^{-1}$ person$^{-1}$) & \small \cite{Capone}\\
\small $d$ & \small Bacteria death rate & \small0.33 (day$^{-1}$) & \small \cite{Capone}\\
\small $\rho$ & \small Contact rate & \small 0.01891 (cell/ml day$^{-1}$ person$^{-1}$)  
& \small Hypothetical\\
\small $S(0)=S_0$ & \small Susceptible individuals at $t=0$ & \small 28 249 670 (person) & \small \cite{Yemen_pop} \\
\small $I(0)=I_0$ & \small Infected individuals at $t=0$ & \small 750 (person) & \small \cite{WHO}\\
\small $Q(0)=Q_0$ & \small Quarantined individuals at $t=0$ & \small 0 (person) & \small Hypothetical \\
\small $R(0)=R_0$ & \small Recovered individuals at $t=0$ & \small 0 (person) & \small Hypothetical \\
\small $B(0)=B_0$ & \small Bacterial concentration at $t=0$ & \small$275\times10^3$ (cell/ml) & \small Hypothetical\\
\hline
\end{tabular}
\end{center}
\end{table}

% --------------------

\subsection{Optimal Solution in Case of Low Resources}
\label{subsec:OC_low}

We start by assuming that $u_{\max}=0.20$, that is, the maximum percentage 
of susceptible individuals that have access to the CWT is $20\%$. 
As we consider that $t_f = 354$ days, then the number of grid points is
$N = 100 \times t_f = 35400$. The numerical simulations for the control
are in agreement with Theorem~\ref{Cap7:theo:nc}. Note that
the values of the control $u$ 
decrease to zero at $t_f = 354$ (see bottom right plot of Figure~\ref{fig:X_u_02}). 
However, all resources are being used during almost all the time period 
considered in the simulation (354 days).
The state trajectories associated with the extremal control are plotted 
in Figure~\ref{fig:X_u_02}. 
%----------------------------------------------------------------------
\begin{figure}[!htb]
\begin{center}
\subfloat
{\includegraphics[scale=0.36]{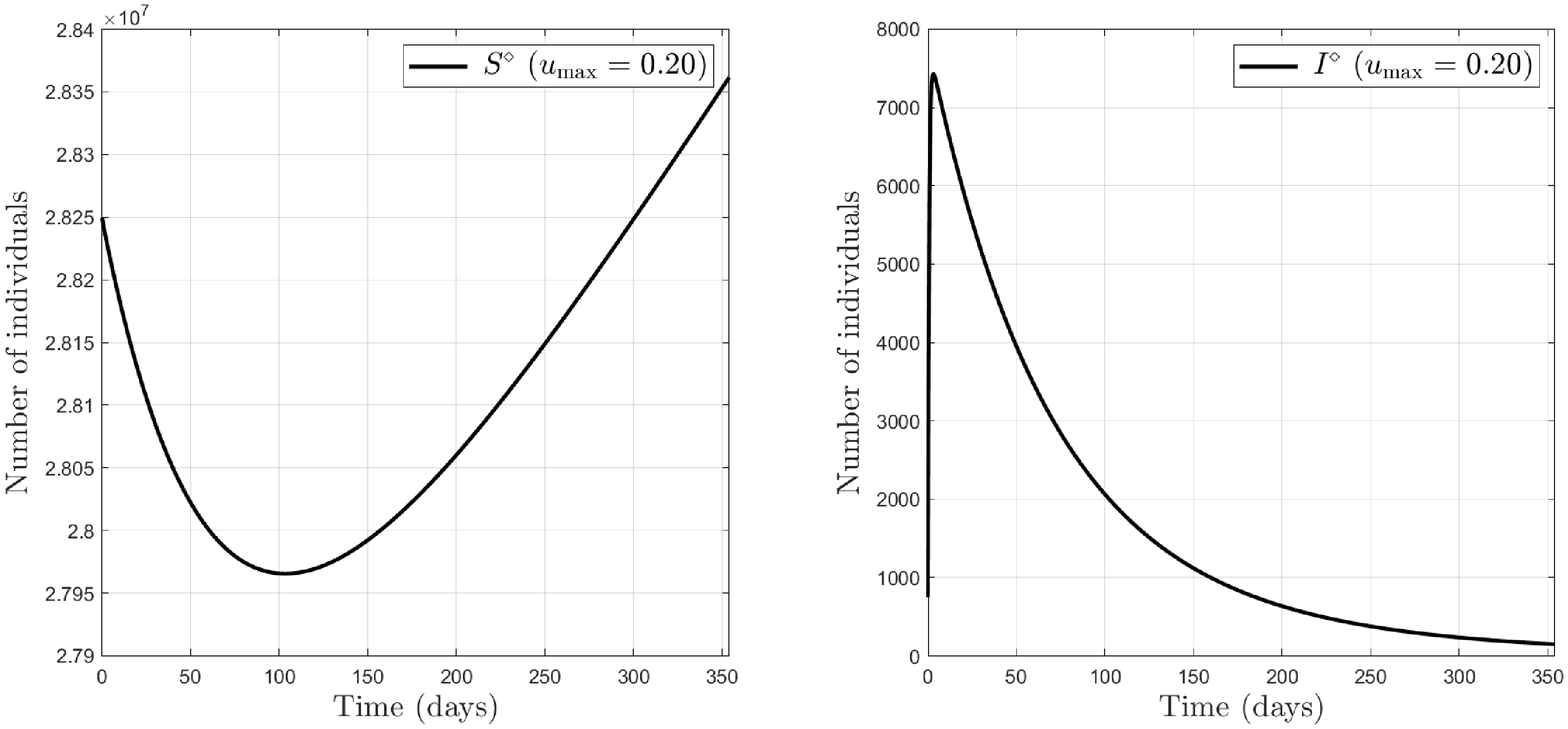}}
\qquad
\subfloat
{\includegraphics[scale=0.36]{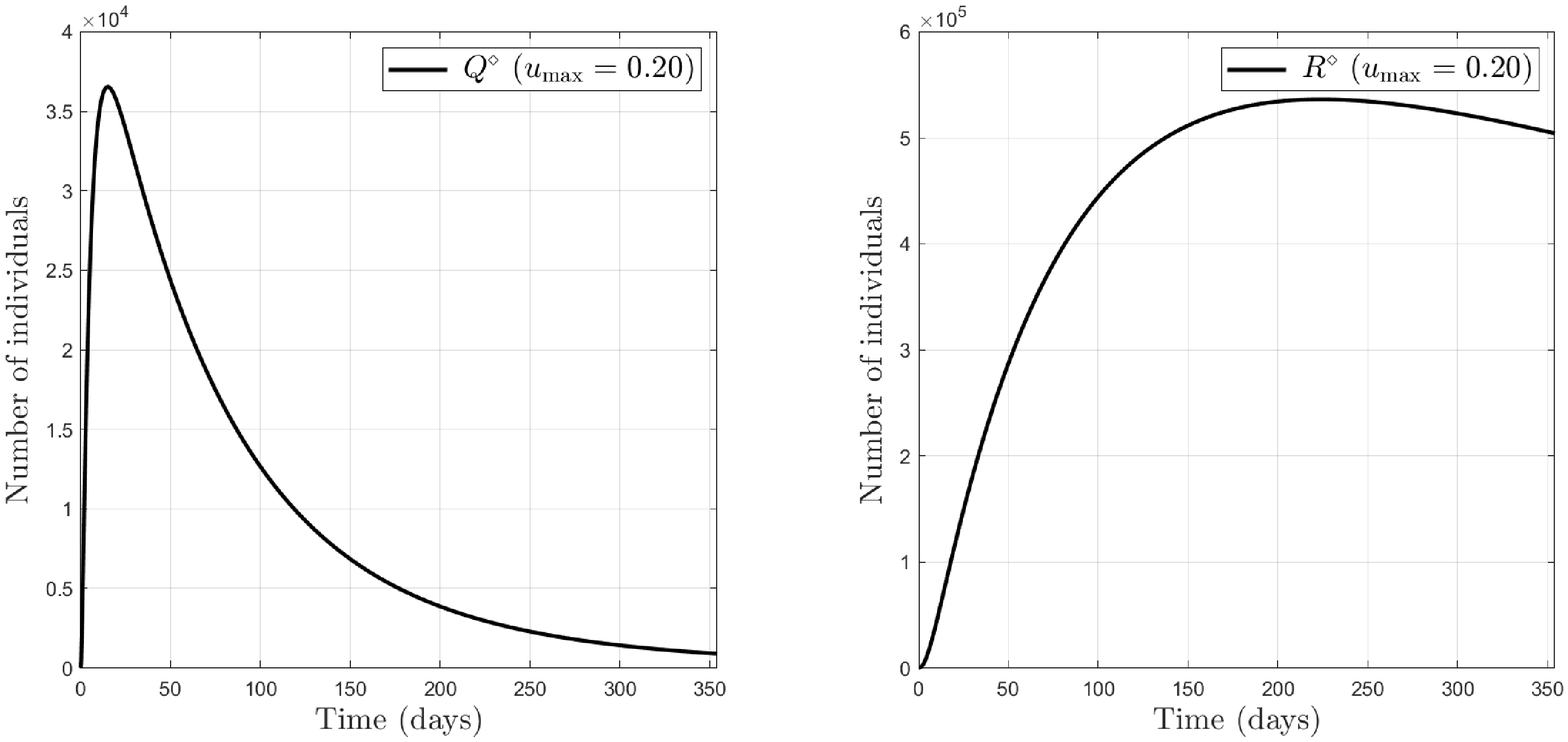}}
\qquad
\subfloat
{\includegraphics[scale=0.36]{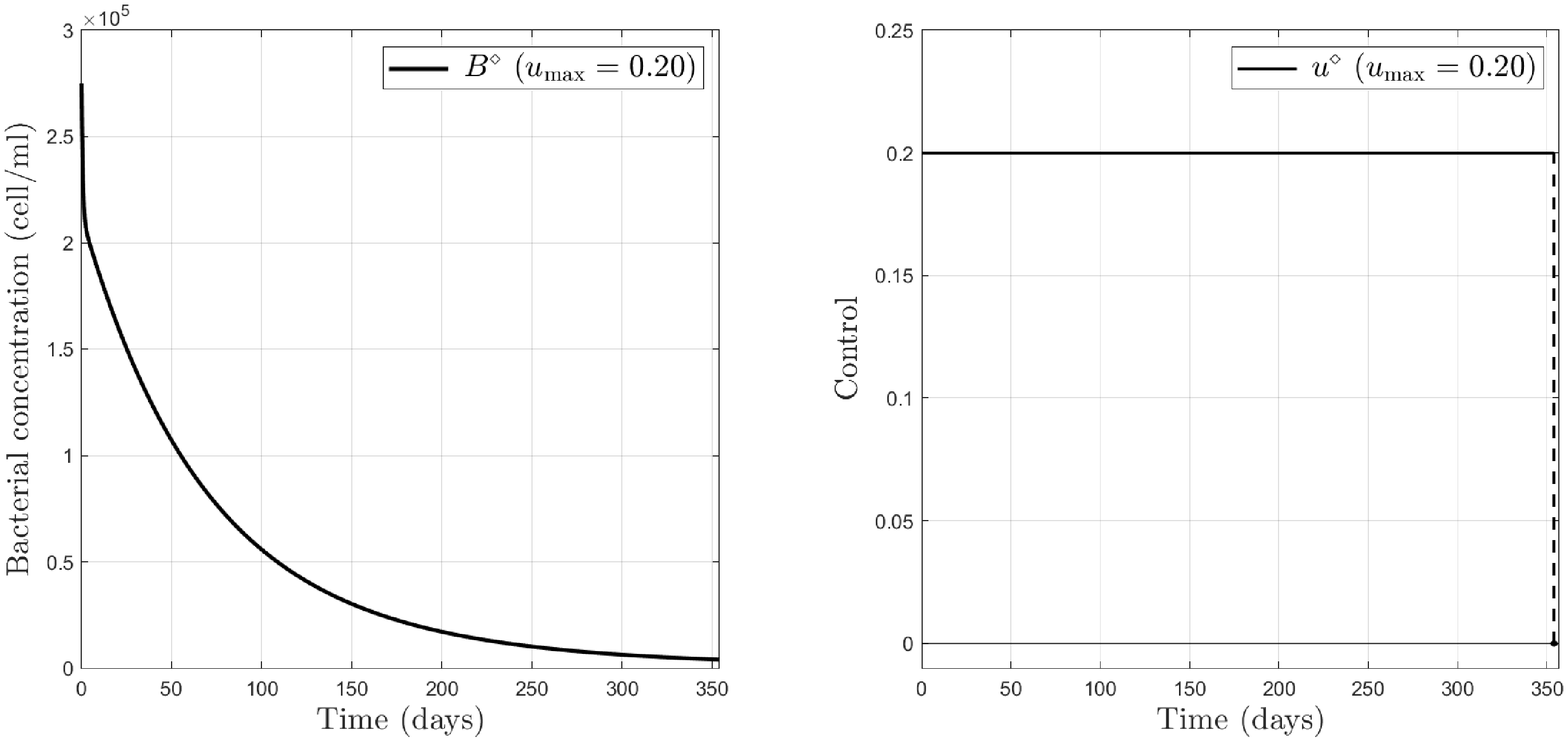}}
\end{center}
\begin{center}
\begin{minipage}{0.7\textwidth}
\caption{\small Extremal state trajectories 
$S^\diamond(t)$, $I^\diamond(t)$, $Q^\diamond(t)$, $R^\diamond(t)$ and $B^\diamond(t)$ 
and extremal control $u^\diamond(t)$ (satisfying the control law 
\eqref{control_law:Ezio}) associated with optimal control problem 
\eqref{mincostfunct} for all $t\in[0,354]$ and $u_{\max}=0.20$, 
using all the other values of Table~\ref{Tab_Parameter}.}
\label{fig:X_u_02}
\end{minipage}
\end{center}
\end{figure}
%-----------------------------------------
From this last figure, we observe that 
although $B$ is a strictly decreasing function, $I$ is not. 
Although the resources are insufficient 
to eradicate the disease, in the considered time interval, the distribution of
CWT to 20\% of the susceptible population is enough to improve the real situation represented 
in Figure~\ref{fig:Yemen}, decreasing significantly the maximum number 
of infective individuals. Note that the real maximum number of infective
individuals was $51\ 000$ and the one associated with $u_{\max}=0.20$ 
is approximately equal to $7\ 431$, an important improvement.

% ---------------------------------------------------------------------------

\subsection{Optimal Solution in Case of Sufficient Resources}
\label{subsec:OC_enough}

As we may deduce from the previous Subsection~\ref{subsec:OC_low}, 
to curtail the spread of the epidemic more quickly and in a better way, 
we need to consider larger values for $u_{\max}$. Now, let us take
$u_{\max}=0.55$ and $u_{\max}=0.90$. In the first case, a little 
bit more than half of susceptible individuals receives the chlorine 
water tablets ($u_{\max}=55\%$). In the second one, only 10\% of 
the susceptible population does not have access to CWT ($u_{\max}=90\%$).
	
Even considering these larger values for $u_{\max}$, the solution 
of infective individuals does not become a strictly decreasing function: 
neither when $u_{\max}=0.55$ nor when $u_{\max}=0.90$. Nevertheless, 
the maximum value of infective individuals decreases significantly 
with respect to the one obtained in Subsection~\ref{subsec:OC_low}. 
Here this value is approximately equal to $3\ 749$ and $942$ 
for $u_{\max}=0.55$ and $u_{\max}=0.90$, respectively 
(see Table~\ref{Tab_Parameter_comparision}).
	
The extremal control is bang-bang for both values of $u_{\max}$. 
We need approximately $82$ days to solve the epidemic, 
when $u_{\max}=0.55$. Thus, at the end of approximately twelve 
weeks, the supply of CWT to 
susceptible population can be discontinued, because the control decreases to zero.
As we expected, one needs less time to curtail the spread of the epidemic when we consider 
$u_{\max}=0.90$: at the end of approximately $40$ days, the control 
decreases to zero and the disease is eradicated (see Table~\ref{Tab_Parameter_comparision}).

The Pontryagin Maximum Principle is a first order necessary 
optimality condition. Therefore, the control law given by 
\eqref{control_law:Ezio} is just an extremal of the optimal control 
problem \eqref{ModeloColeraControlo}--\eqref{mincostfunct}. 
However, a stronger condition, the so-called 
strict bang-bang property proved in \cite{Osmolovskii}, 
is also satisfied for $u_{\max}=0.55$ and $u_{\max}=0.90$, 
that is, the bang-bang control and the switching function 
match the following switching conditions:
\begin{equation*}
\begin{split}
&\phi_{55}(t)<0,\ \textrm{if}\ 0\leq t <t_s^{55},\\
&\dot{\phi}_{55}(t_s^{55}) \simeq 0.228049 > 0,\\
&\phi_{55}(t)>0,\ \textrm{if}\ t_s^{55} < t \leq 100,
\end{split}
\quad\textrm{and}\quad\quad
\begin{split}
&\phi_{90}(t)<0,\ \textrm{if}\ 0\leq t <t_s^{90},\\
&\dot{\phi}_{90}(t_s^{90}) \simeq 0.389516 > 0,\\
&\phi_{90}(t)>0,\ \textrm{if}\ t_s^{90} < t \leq 70,
\end{split}
\end{equation*}
where $t_s^p$ and $\phi_p$ denote, respectively, the switching time and switching function $\phi$ 
for $u_{\max}=\frac{p}{100}$. Moreover, the respective minimum costs are given by
$$
J_{55}\simeq 3.797326 \times 10^4 \quad\textrm{and}
\quad J_{90}\simeq 4.443881\times 10^3.
$$

% ---------------------------------------------------------------------------

\subsection{Optimal Solution in Case of Abundance of Resources}
\label{subsec:OC_abundance}

In this subsection, we consider $u_{\max}=0.95$, that is, $95\%$ 
of susceptible population has access to CWT for water purification, 
corresponding to a situation where there is abundance of resources. 
In this case, the numerical solution for the number of infective 
individuals $I$ is a strictly decreasing function. In this situation, there is a timely 
and effective distribution of CWT, which avoids the increase of the number
of infected individuals. Consequently, it is possible to achieve 
a low maximum value of infected individuals equal to $I_0=750$. 
	
When $u_{\max}=0.95$, we only need to distribute CWT in the 
first $37$ days. 
The minimum cost takes the value $J_{95}\simeq 2.099780\times10^3$
and the extremal control is also bang-bang for $u_{\max}=0.95$.
As in Subsection~\ref{subsec:OC_enough}, the bang-bang 
control and the switching function match the switching condition 
\eqref{control_law:Ezio} and satisfy the strict bang-bang property 
with respect to the Pontryagin Maximum Principle \cite{Osmolovskii}:
\begin{equation*}
\begin{split}
\phi_{95}(t)<0,\ \textrm{if}\ 0\leq t <t_s^{95},\\
\dot{\phi}_{95}(t_s^{95}) \simeq 0.418741 > 0,\\
\phi_{95}(t)>0,\ \textrm{if}\ t_s^{95} < t \leq 70,
\end{split}
\end{equation*}
where $t_s^{95}$ and $\phi_{95}$ denote, respectively, the switching time and switching function $\phi$ 
for $u_{\max}=\frac{95}{100}$.
We compare the switching times and the total number of infected individuals 
associated with Subsections~\ref{subsec:OC_low}, \ref{subsec:OC_enough} and 
\ref{subsec:OC_abundance} in Table~\ref{Tab_Parameter_comparision}.
% ------------------------------------------------------
\begin{table}[ht]
\begin{minipage}{0.78\textwidth}
\caption{\small Switching time and total number of infected individuals for all considered cases.}
\label{Tab_Parameter_comparision}
\end{minipage}
\begin{center}
\begin{tabular}[center]{|c|c|c|c|c|} \hline
& $u_{\max}=0.2$ & $u_{\max}=0.55$ & $u_{\max}=0.9$ & $u_{\max}=0.95$ \\ \hline \hline
\small Switching time (days) & \small $353.99$ & \small $81.91$ & \small $40.24$ & \small $37.35$ \\
\small Total number of infected individuals & \small $7\ 431$ & \small $3\ 749$ & \small $942$ & \small $750$\\ \hline
\end{tabular}
\end{center}
\end{table}
% ------------------------------------------------------

\subsection{Local Asymptotic Stability of the Endemic Equilibrium}
\label{subsec:stability:EE}

For the parameter values in Table~\ref{Tab_Parameter}, we have that the basic reproduction
number \eqref{R0} is
$$R_0\simeq 3.830175$$
and that the endemic equilibrium \eqref{EndemicEquilibrium} is, approximately,
\begin{equation}
\label{EE:NSimulation}
E^*=\left(
2.900036\times10^7,\ 
1.039755\times10^5,\ 
5.978021\times10^5,\ 
1.075290\times10^8,\ 
2.788426\times10^6\ 
\right).
\end{equation}
Plotting the state trajectories predicted by model~\eqref{ModeloColera} for the values of 
Table~\ref{Tab_Parameter}, we can observe, numerically, the local asymptotic stability of the 
endemic equilibrium \eqref{EE:NSimulation}: see Figure~\ref{Fig:EE:stab}.
% --------------------------------------
\begin{figure}[!htb]
\begin{center}
\subfloat{
\includegraphics[scale=0.36]{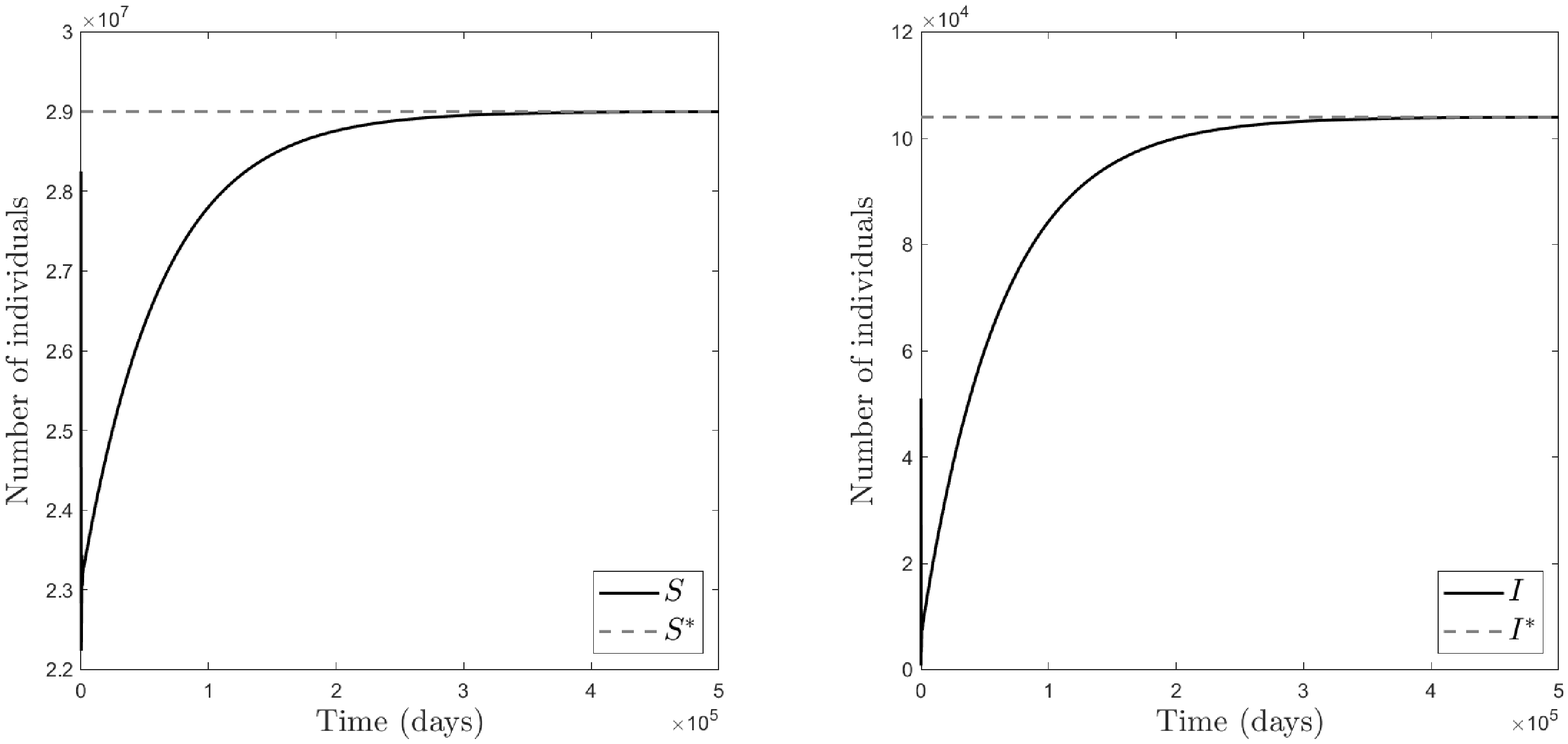}}
\qquad
\subfloat{
\includegraphics[scale=0.36]{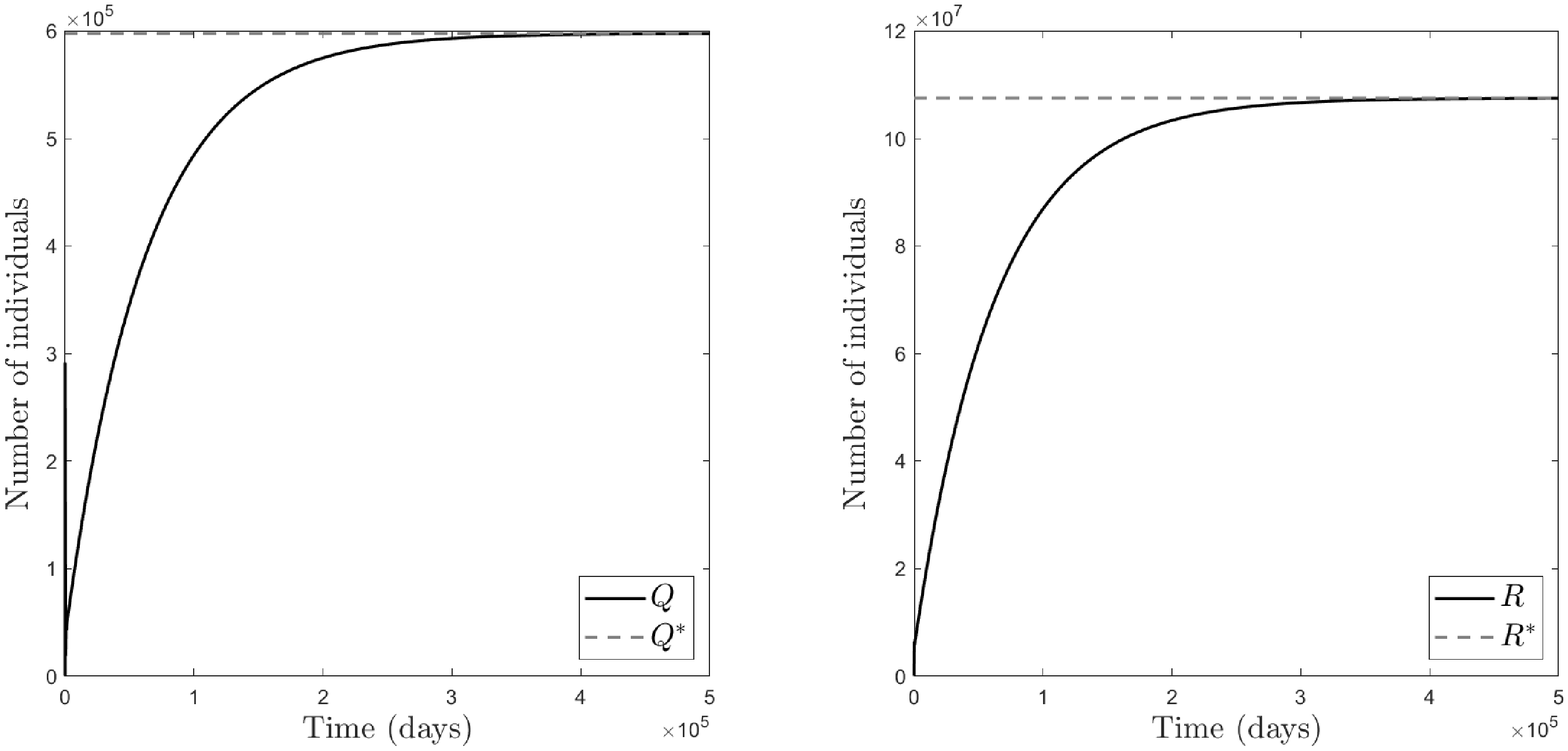}}
\qquad
\subfloat{
\includegraphics[scale=0.34]{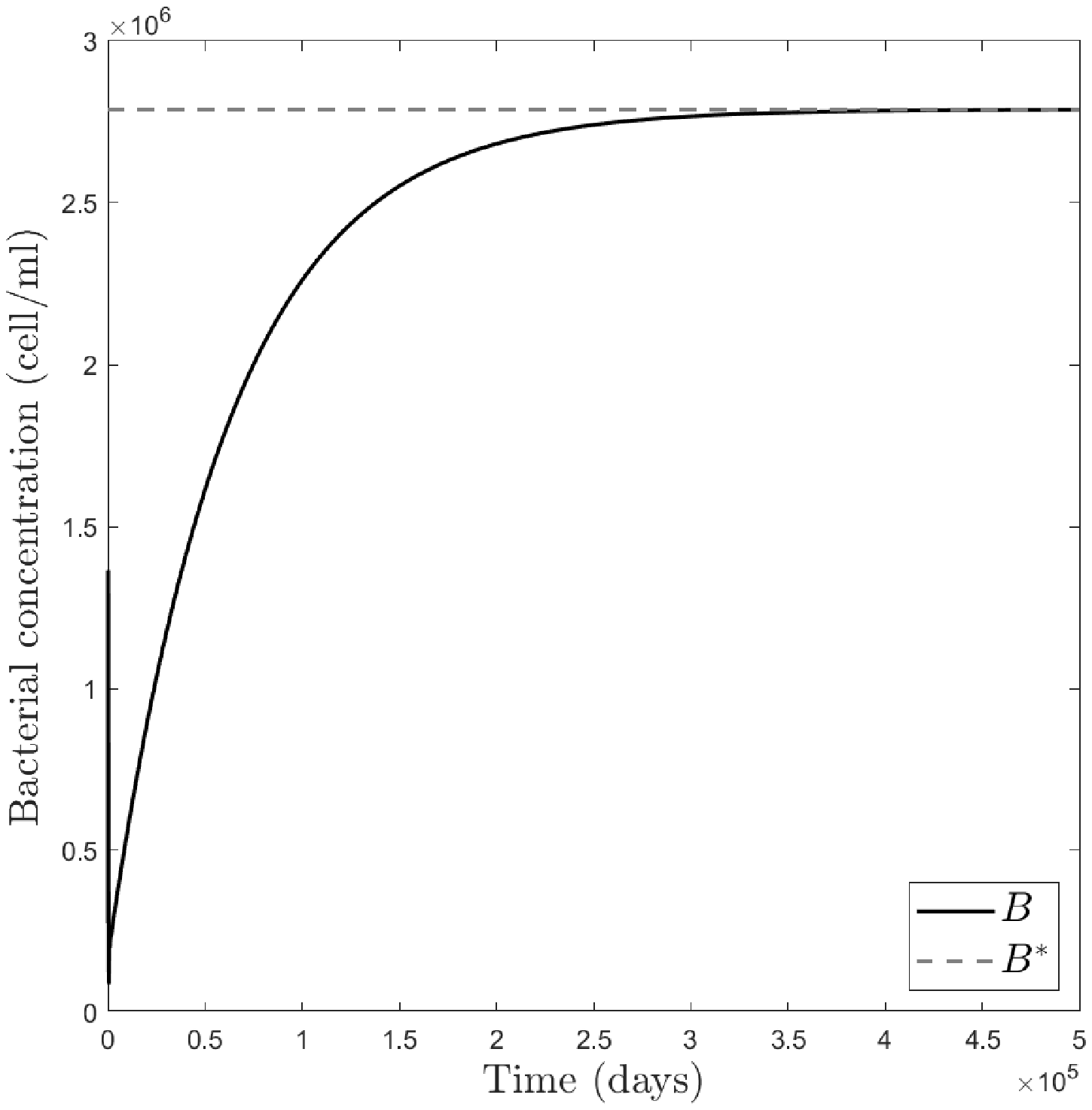}}
\end{center}
\begin{center}
\begin{minipage}{0.7\textwidth}
\caption{\small State trajectories of model~\eqref{ModeloColera} 
versus the endemic equilibrium \eqref{EE:NSimulation} 
for $T=5\times10^5$ days and considering all the other values 
of Table~\ref{Tab_Parameter}.}
\label{Fig:EE:stab}
\end{minipage}
\end{center}
\end{figure}
%-----------------------------------------------

\section{Conclusions}
\label{sec:conc}

In this paper, we improved the mathematical model proposed in \cite{LPST} 
by incorporating the requirement that a healthy individual must intake 
bacteria from the environment to become infected and, by doing so, 
these bacteria are removed from the aquatic medium. 
In contrast to \cite{LPST}, the feasibility of the endemic 
equilibrium depends on the rate at which the bacteria are spread 
by the infective, and must exceed the combined rates 
at which infective
leave their compartment, i.e., must be larger than the sum of the 
rates at which individuals are quarantined, and die either naturally 
or because of the disease. The conditions for the local stability 
of the endemic equilibrium also differ from the ones of \cite{LPST}. 

We proposed and analyzed an optimal control problem, where the control 
function represents the fraction of susceptible population who receive 
chlorine water tablets (CWT) for water purification, 
with the objective to minimize the number of infective
individuals as well as the cost associated with the distribution of CWT. The optimal 
solution has been characterized both analytical and numerically. 
The extremal bang-bang controls satisfy the so-called strict bang-bang 
property with respect to the Pontryagin Maximum Principle. 
Thus, the proposed strategies for the distribution of CWT
represent suitable means for containing cholera outbreaks,
in different scenarios and periods of time. This is supported
by the current situation in Mozambique, where the Portuguese army 
purifies around $4\ 000$ liters of water per day using chlorine, to fight
the cholera epidemic caused by the passage of cyclone Idai in March 2019 \cite{moz}.

In our work, we assume a homogeneously mixing population and the distribution of CWT 
to susceptible individuals is done randomly. 
As future work, it would be interesting to propose a model defined by 
partial differential equations in order to consider a temporal 
and spatial distribution of CWT.
Studying a model that incorporates a spatial distribution, 
we could decide to distribute CWT only to susceptible individuals who live in areas 
that can be more easily reached by health workers.
Moreover, the proposed model could be generalized by considering seasonality 
(see, e.g., \cite{Buonomo,Pascual2,Pascual, Pourabbas} and references cited therein).
Another line of research consists to find how the optimal control and its results are influenced
by the existing uncertainties on the parameters of the model. 
That would be an important message for health
authorities and will be addressed elsewhere.

% -------------------------------------
 
\begin{acknowledgements}
This research was supported by the
Portuguese Foundation for Science and Technology (FCT)
within projects UIDB/04106/2020 and UIDP/04106/2020 (CIDMA) 
and PTDC/EEI-AUT/2933/2014 (TOCCATA), funded by Project 
3599 -- Promover a Produ\c{c}\~ao Cient\'{\i}fica e Desenvolvimento
Tecnol\'ogico e a Constitui\c{c}\~ao de Redes Tem\'aticas 
and FEDER funds through COMPETE 2020, Programa Operacional
Competitividade e Internacionaliza\c{c}\~ao (POCI).
Lemos-Pai\~{a}o is also supported by the Ph.D.
fellowship PD/BD/114184/2016; Silva 
by national funds (OE), through FCT, I.P., in the scope of the 
framework contract foreseen in the numbers 4, 5 and 6 of the
article 23, of the Decree-Law 57/2016, of August 29, 
changed by Law 57/2017, of July 19. \uppercase{T}he 
research of \uppercase{E}zio \uppercase{V}enturino 
has been partially supported by the project
``\uppercase{M}etodi numerici e computazionali per le scienze applicate'' 
of the \uppercase{D}ipartimento di \uppercase{M}atematica
``\uppercase{G}iuseppe \uppercase{P}eano''.
The authors are very grateful to two anonymous referees for several constructive
remarks and questions that helped them to improve the quality of the paper.
\end{acknowledgements}

% ------------------------------------------

% ------------------------------------------

\end{document}